\documentclass[reqno]{amsart}
\usepackage{amssymb,amsmath,amsthm,amscd,latexsym,amsfonts}
\usepackage{mathtools}
\usepackage[T1]{fontenc}

\usepackage{cite}

\usepackage{appendix}

\usepackage{times}

\newtheorem{pr}{Proposition}

\newtheorem{de}{Definition}
\newtheorem{teo}{Theorem}

\newfont{\hueca}{msbm10}

\begin{document}
\title[Minimal linear representations of filiform Lie algebras]
{Minimal linear representations of filiform Lie algebras and their application for construction of Leibniz algebras}

\author{I.A.~Karimjanov, M.~Ladra  }
\address{[I.A. Karimjanov--M. Ladra] Department of Algebra, University of Santiago de Compostela, 15782, Santiago de Compostela, Spain}
\email{iqboli@gmail.com -- manuel.ladra@usc.es}

\thanks{The work was partially supported  was supported by Ministerio de Econom\'ia y Competitividad (Spain),
grant MTM2013-43687-P (European FEDER support included) and  by Xunta de Galicia, grant GRC2013-045 (European FEDER support included).}

\begin{abstract}
In this paper we find minimal faithful representations of several classes filiform Lie algebras by means of strictly upper-triangular matrices. We investigate Leibniz algebras whose corresponding Lie algebras are filiform Lie algebras such that the action $I \times L \to I$ gives rise to a minimal faithful representation of  a filiform Lie algebra. The classification up to isomorphism of such Leibniz algebras is given for  low-dimensional cases.
\end{abstract}


\maketitle

\section{Introduction}

According Ado's Theorem, given any finite-dimensional complex Lie algebra $\mathfrak{g}$, there exists
a matrix algebra isomorphic to $\mathfrak{g}$. In this way, every finite-dimensional complex Lie algebra can be
represented as a Lie subalgebra of the complex general linear algebra $\mathfrak{gl}(n,\mathbb{C})$, formed by
all the complex $n\times n$ matrices, for some $n \in \mathbb{N}$. We
consider the following integer valued invariant of $\mathfrak{g}$:
$$ \mu(\mathfrak{g}) = \min\{ \dim(M) \mid M \ \text{is a faithful} \ \mathfrak{g}\text{-module} \}$$
It follows from the proof of Ado's Theorem that $\mu(\mathfrak{g})$ can be bounded by a function
depending on only $n$. This value is also equal to the minimal value $n$ such that $\mathfrak{gl}(\mathbb{C},n)$ contains a subalgebra
isomorphic to $\mathfrak{g}$.

Given a Lie algebra $\mathfrak{g}$, a representation of $\mathfrak{g}$ in $\mathbb{C}^n$ is a homomorphism of Lie algebras
$f \colon \mathfrak{g} \rightarrow \mathfrak{gl}(\mathbb{C}^n ) = \mathfrak{gl}(n,\mathbb{C})$. The natural integer $n$ is called the dimension of this representation. We consider faithful representations because such representations
allow us to identify a given Lie algebra with its image under the representation, which is
a Lie subalgebra of $\mathfrak{gl}(n,\mathbb{C})$.
Representations can be also defined by using arbitrary $n$-dimensional vector spaces $V$
(see \cite{FuHa91}). In such a case, a representation would be a homomorphism of Lie algebras from $\mathfrak{g}$
to the Lie algebra  of the endomorphisms of the vector space $V$, $\mathfrak{gl}(V )$, which is called a $\mathfrak{g}$-module.
However, it is sufficient to consider representations on $\mathbb{C}^n$ because there always exists a
unique $n \in \mathbb{N}$ such that $V$ is isomorphic to $\mathbb{C}^n$ .

Many works are devoted to find the value $\mu(\mathfrak{g})$ of several finite-dimensional Lie algebras. In \cite{Bur98},
 the value of $\mu(\mathfrak{g})$ for abelian Lie algebras and Heisenberg algebras is found, and moreover the estimated value of $\mu(\mathfrak{g})$ for filiform Lie algebras is given. In the works  \cite{BNT08,CNT13,GhTh15} the authors find the matrix representation of some low-dimensional Lie algebras.

In paper \cite{CNT13} the minimal faithful representation of the  filiform Lie algebra $\mathcal{L}_n$ is shown and the  authors denote $$\overline{\mu}(\mathfrak{g}) = \min\{ n \in \mathbb{N} \mid  \text{subalgebra of} \ g_n \ \text{isomorphic to} \ \mathfrak{g}\},$$ where $g_n$ is an upper triangular square matrix of dimension $n$. Moreover, they prove the next proposition
\begin{pr}[\cite{CNT13}] \label{eq} Let $\mathfrak{g}$ be an $n$-dimensional filiform Lie algebra. Then $\overline{\mu}(\mathfrak{g})\geq n$.
\end{pr}

 The paper is devoted to find minimal linear representations of some classes of filiform Lie algebras of dimension $n$. Exactly we find a minimal faithful representation of the filiform Lie algebras $\mathcal{Q}_{2n}, \mathcal{R}_{n}$ and $\mathcal{W}_n$. Moreover we construct Leibniz algebras using these representations of filiform Lie algebras.

 Leibniz algebras, which
are a non-antisymmetric generalization of Lie algebras, were introduced in 1965
by Bloh in \cite{Bloh}, who called them $D$-algebras and in 1993 Loday  \cite{Lod} made them
popular and studied their (co)homology.

\begin{de} An algebra $(L,[-,-])$ over a field  $\mathbb{F}$   is called a Leibniz algebra if for any $x,y,z\in L$, the so-called Leibniz identity
\[ \big[[x,y],z\big]=\big[[x,z],y\big]+\big[x,[y,z]\big] \] holds.
\end{de}

One of the method of classification Leibniz algebras is the study of Leibniz algebras with given corresponding Lie algebras.
In the papers \cite{ACK15,CCO16,ORT13,UKO15},  Leibniz algebras whose corresponding Lie algebras are naturally graded filiform Lie algebras $L_n$, Heisenberg algebras, simple Lie $\mathfrak{sl}_2$  and Diamond Lie algebras are studied.
Let ${ L}$ be a Leibniz algebra. The ideal $I$ generated by the squares of elements of the algebra $L$, that is by the set $\{[x,x]: x\in {L}\}$, plays an important role in the theory since it determines the (possible) non-Lie character of ${L}$. From the Leibniz identity, this ideal satisfies
$$[{L},I]=0.$$
Clearly, the quotient algebra $L / I$ is a Lie algebra, called the {\it corresponding Lie algebra} of $L$. The map $I \times L / I \to I$, $(i,\overline{x}) \mapsto [i,x]$ endows $I$ with a structure of $L / I$-module (see \cite{AAO05}).

Denote by $Q(L) = L / I \oplus I$. Then the operation $(-,-)$ defines a Leibniz algebra structure on $Q(L),$ where $$(\overline{x},\overline{y}) = \overline{[x,y]}, \quad (\overline{x},i) = [x,i],
\quad (i, \overline{x}) = 0, \quad (i,j) = 0, \qquad x, y \in L, \ i,j \in I.$$

Therefore, given a Lie algebra $G$ and a $G$-module $M,$ we can construct a Leibniz algebra $(G, M)$ by the above construction. The main problem which occurs in this connections is a description of a Leibniz algebra $L$, such that the corresponding Leibniz algebra $Q(L)$ is isomorphic to a priory given algebra $(G, M)$.

Now we give definitions of nilpotent and filiform Lie algebras.

For a Lie algebra $L$ consider the following  lower central series:
\[L^1=L,\qquad L^{k+1}=[L^1,L^k] \qquad k \geq 1.\]

\begin{de} A Lie algebra $L$ is called nilpotent if there exists  $s\in\mathbb N $ such that $L^s=0$.
\end{de}

\begin{de} A Lie algebra $L$ is said to be filiform if $\dim L^i=n-i$, where $n=\dim L$ and $2\leq i \leq n$.
\end{de}

We list some classes of $n$-dimension filiform Lie algebras with  basis $\{e_1,\dots,e_n\}$.

1. Let $\mathcal{L}_n$ be the Lie algebra defined by
$$[e_1,e_i]=-[e_i,e_1]=e_{i+1}, \quad 2\leq i \leq n-1.$$

2. Let $\mathcal{Q}_{2s} \ (n=2s)$ be the nilpotent Lie algebra defined  by
\begin{align*}
[e_1,e_i] & =  -[e_i,e_1]=e_{i+1}, && 2\leq i \leq 2s-2, \\
[e_{2s+1-i},e_i]  & =  -[e_i,e_{2s+1-i}]=(-1)^i\,e_{2s},&& 2\leq i
\leq s.\end{align*}

3. Let $\mathcal{R}_{n}$ be defined  by
$$\begin{array}{ll}
[e_1,e_i] =  -[e_i,e_1]=e_{i+1},& 2\leq i \leq n-1,\\[1mm]
 [e_2,e_i] =  -[e_i,e_2]=e_{i+2},& 3\leq i
\leq n-2.\end{array}$$

4. Let $\mathcal{W}_n$ be the Lie algebra whose brackets in the basis are:
$$[e_i,e_j]=-[e_j,e_i]=(j-i)e_{i+j}, \quad i+j\leq n.$$

The algebras $\mathcal{L}_n$ and $\mathcal{Q}_{2n}$ are naturally graded filiform Lie algebras. The algebra $\mathcal{W}_n$ is the finite-dimensional Witt algebra.

\section{Minimum linear representation of filiform Lie algebras}

\begin{pr}
Let $\mathcal{Q}_{2n}$ be a $2n$-dimensional filiform Lie algebra with basis $\{e_i\}^{2n}_{i=1}$. Then its minimal faithful representation is given by
$$a_1e_1+a_2e_2+\dots+a_{2n}e_{2n}\mapsto\left(\begin{matrix}
0&a_2&-a_3&\dots&a_{2n-2}&-a_{2n-1}&-2a_{2n}\\
0&0&a_1&\dots&0&0&a_{2n-1}\\
0&0&0&\dots&0&0&a_{2n-2}\\
\vdots&\vdots&\vdots&\ddots&\vdots&\vdots&\vdots\\
0&0&0&\dots&0&a_1&a_3\\
0&0&0&\dots&0&0&a_2\\
0&0&0&\dots&0&0&0
\end{matrix}\right).$$
\end{pr}

\begin{proof}
Consider a bilinear map $\varphi \colon Q_{2n}\rightarrow\mathfrak{gl}_{2n}$ given by
$$\varphi(e_1) = \sum\limits_{k=2}^{2n-2}E_{k,k+1},\quad \varphi(e_i) = (-1)^iE_{1,i}+E_{2n-i+1,2n}\quad 2 \leq i \leq 2n-1, \quad \varphi(e_{2n})=-2E_{1,2n},$$
where $E_{i,j}$ is the matrix with  ($i, j$)-th entry equal to $1$ and others zero.

Checking $[\varphi(e_i),\varphi(e_j)]=\varphi(e_i)\varphi(e_j)-\varphi(e_j)\varphi(e_i)$ for all  $1\leq i,j \leq 2n,$ we verify that $\varphi$ is an isomorphism of algebras. Then by Proposition \ref{eq} we obtain that it is minimal.
\end{proof}

Let us denote by $V=\mathbb{C}^{2n}$ the natural $\varphi(Q_{2n})$-module and endow it with a $Q_{2n}$-module structure by  $$(x, e) =  x \varphi(e).$$ Then we obtain
\begin{equation}\label{eq1}\left\{\begin{array}{ll}(x_i, e_1) = x_{i+1}, & 2 \leq i \leq 2n-2,\\[1mm]
(x_1,e_i) = (-1)^i x_i, & 2 \leq i \leq 2n-1,\\[1mm]
(x_{2n+1-i}, e_i) = x_{2n}, & 2 \leq i \leq 2n-1,\\[1mm]
(x_{1},e_{2n}) = -2x_{2n}, &
\end{array}\right.\end{equation}
and the remaining products are zero.

\begin{pr}
Let $\mathcal{R}_{n}$ be a $n$-dimensional filiform Lie algebra with basis $\{e_i\}^{n}_{i=1}$. Then its minimal faithful representation is given by
$$a_1e_1+a_2e_2+\dots+a_{n}e_{n}\mapsto\left(\begin{matrix}
0&a_1&a_2&0&\dots&0&0&0&a_{n}\\
0&0&a_1&a_2&\dots&0&0&0&a_{n-1}\\
0&0&0&a_1&\dots&0&0&0&a_{n-2}\\
0&0&0&0&\dots&0&0&0&a_{n-3}\\
\vdots&\vdots&\vdots&\vdots&\ddots&\vdots&\vdots&\vdots&\vdots\\
0&0&0&0&\dots&0&a_1&a_2&a_4\\
0&0&0&0&\dots&0&0&a_1&a_3\\
0&0&0&0&\dots&0&0&0&a_2\\
0&0&0&0&\dots&0&0&0&0
\end{matrix}\right).$$
\end{pr}

\begin{proof}
We take bilinear map $\psi \colon \mathcal{R}_{n} \rightarrow \mathfrak{gl}(n,\mathbb{C})$ given by
$$\psi(e_1) = \sum\limits_{i=1}^{n-2}E_{i,i+1},\quad \psi(e_2)=\sum\limits_{i=1}^{n-3}E_{i,i+2}+E_{n-1,n}, \quad \psi(e_i) = E_{n+1-i,n}, \quad 3 \leq i \leq n.$$

Checking $[\psi(e_i),\psi(e_j)]=\psi(e_i)\psi(e_j)-\psi(e_j)\psi(e_i)$ for all  $1\leq i,j \leq n,$ we verify that $\psi$ is an isomorphism of algebras. Then by Proposition \ref{eq} we obtain that it is minimal.\end{proof}

Now, we construct a module $V \times \mathcal{R}_{n} \rightarrow V,$ such that $$(x, e) =  x \varphi(e) .$$ Then we obtain
\begin{equation*} \left\{\begin{array}{ll}
(x_{i}, e_1) = x_{i+1}, & 1 \leq i \leq n-2,\\[1mm]
(x_i, e_2) = x_{i+2}, & 1 \leq i \leq n-3,\\[1mm]
(x_{n+1-j}, e_j) = x_n, & 2 \leq j \leq n.
\end{array}\right.\end{equation*}
the remaining products in the action being zero.

Denote  by $C^n_{m}=\binom{m}{n}$ the binomial coefficient.
\begin{pr}
Let $\mathcal{W}_{n}$ be an $n$-dimensional filiform Lie algebra with basis $\{e_i\}^n_{i=1}$. Then $\mathcal{W}_{n}$ is isomorphic to a subalgebra of $\mathfrak{gl}(n,\mathbb{C})$ by $\varphi:$
$$\varphi(e_1) = \sum\limits_{k=1}^{n-2}E_{k,k+1},\quad \varphi(e_2)=\sum\limits_{k=1}^{n-3}\frac{1}{n-k}E_{k,k+2}+E_{n-1,n},$$
$$\varphi(e_i)=\frac{1}{(i-2)!}\Big(\sum\limits_{k=1}^{n-i-1}\big(\sum\limits_{s=0}^{i-2}\frac{(-1)^{i+s}C^s_{i-2}}{n-k-s}\big)E_{k,k+i}+E_{n+1-i,n}\Big), \quad 3\leq i\leq n,$$
and this faithful representation is minimal.
\end{pr}

\begin{proof}
We take the isomorphism $\varphi \colon \mathcal{W}_{n} \rightarrow \mathfrak{gl}(n,\mathbb{C})$ such that
$$\varphi(e_1) = \sum\limits_{k=1}^{n-2}E_{k,k+1},\quad \varphi(e_2)=\sum\limits_{s=1}^{n-3}\alpha_sE_{s,s+2}+E_{n-1,n}.$$

Now we consider

$$\varphi(e_3)=[\varphi(e_1),\varphi(e_2)]=\varphi(e_1)\varphi(e_2)-\varphi(e_2)\varphi(e_1)=\big(\sum\limits_{k=1}^{n-2}E_{k,k+1}\big) \ \big(\sum\limits_{s=1}^{n-3}\alpha_sE_{s,s+2}+E_{n-1,n}\big)-$$
$$\big(\sum\limits_{s=1}^{n-3}\alpha_sE_{s,s+2}+E_{n-1,n}\big) \ \big(\sum\limits_{k=1}^{n-2}E_{k,k+1}\big)=\sum\limits_{k=1}^{n-4}(\alpha_{k+1}-\alpha_k)E_{k,k+3}+E_{n-2,n},$$

\begin{multline*}
 \varphi(e_4)=\frac{1}{2}[\varphi(e_1),\varphi(e_3)]=\frac{1}{2}(\varphi(e_1)\varphi(e_3)-\varphi(e_3)\varphi(e_1))=\frac{1}{2}\Big(\big(\sum\limits_{s=1}^{n-2}E_{s,s+1}\big) \ \big(\sum\limits_{k=1}^{n-4}(\alpha_{k+1}-\alpha_k)E_{k,k+3}+E_{n-2,n}\big)\\
 - \big(\sum\limits_{k=1}^{n-4}(\alpha_{k+1}-\alpha_k)E_{k,k+3}+E_{n-2,n}\big) \ \big(\sum\limits_{s=1}^{n-2}E_{s,s+1}\big)\Big)=\frac{1}{2}\big(\sum\limits_{k=1}^{n-5}(\alpha_k-2\alpha_{k+1}+\alpha_{k+2})E_{k,k+4}+E_{n-3,n}\big),
\end{multline*}

Let us suppose that
$$\varphi(e_i)=\frac{1}{(i-2)!}\Big(\sum\limits_{p=1}^{n-i-1}\big(\sum\limits_{s=0}^{i-2}(-1)^{i+s}C^s_{i-2}\alpha_{p+s}\big)E_{p,p+i}+E_{n+1-i,n}\Big), \quad 3\leq i\leq n.$$

Let us we suppose the previous equality true for $i=k$ and we will consider for $i=k+1$.
\begin{multline*}
\varphi(e_{k+1})=\frac{1}{k-1}[\varphi(e_1),\varphi(e_k)]=\frac{1}{k-1}\big(\varphi(e_1)\varphi(e_k)-\varphi(e_k)\varphi(e_1)\big)\\
=\frac{1}{k-1}
\Big(\frac{1}{(k-2)!}\big(\sum\limits_{t=1}^{n-2}E_{t,t+1} \big)\big(\sum\limits_{p=1}^{n-k-1}\big(\sum\limits_{s=0}^{k-2}(-1)^{k+s}C^s_{k-2}\alpha_{p+s}\big)E_{p,p+k}+
E_{n+1-k,n}\big) \\
- \frac{1}{(k-2)!}\big(\sum\limits_{p=1}^{n-k-1}\big(\sum\limits_{s=0}^{k-2}(-1)^{k+s}C^s_{k-2}\alpha_{p+s}\big)E_{p,p+k}+
E_{n+1-k,n}\big)\big(\sum\limits_{t=1}^{n-2}E_{t,t+1}\big)\Big)\\
= \frac{1}{(k-1)!}\Big(\sum\limits_{p=1}^{n-k-2}\big(\sum\limits_{s=0}^{k-2}(-1)^{k+s}C^s_{k-2}(\alpha_{p+s+1}-\alpha_{p+s})\big)E_{p,p+k+1}+E_{n-k,n}\Big)\\
= \frac{1}{(k-1)!}\Big(\sum\limits_{p=1}^{n-k-2}\big(\sum\limits_{s=0}^{k-1}(-1)^{k+s+1}C^s_{k-1}\alpha_{p+s}\big)E_{p,p+k+1}+E_{n-k,n}\Big).
\end{multline*}
From the multiplications, where $i+j\leq n$

\begin{multline*}
[\varphi(e_i),\varphi(e_j)]=\varphi(e_i)\varphi(e_j)-\varphi(e_j)\varphi(e_i) \\
=\frac{1}{(i-2)! \ (j-2)!}
\Big(\Big(\sum\limits_{p=1}^{n-i-1}\big(\sum\limits_{s=0}^{i-2}(-1)^{i+s}C^s_{i-2}\alpha_{p+s}\big)E_{p,p+i}+E_{n+1-i,n}\Big)\\
\Big(\sum\limits_{q=1}^{n-j-1}\big(\sum\limits_{r=0}^{j-2}(-1)^{j+r}C^r_{j-2}\alpha_{q+r}\big)E_{q,q+j}+E_{n+1-j,n}\Big)\\-
\Big(\sum\limits_{q=1}^{n-j-1}\big(\sum\limits_{r=0}^{j-2}(-1)^{j+r}C^r_{j-2}\alpha_{q+r}\big)E_{q,q+j}+E_{n+1-j,n}\Big)\\
\Big(\sum\limits_{p=1}^{n-i-1}\big(\sum\limits_{s=0}^{i-2}(-1)^{i+s}C^s_{i-2}\alpha_{p+s}\big)E_{p,p+i}+E_{n+1-i,n}\Big)\Big)\\
=\frac{1}{(i-2)! \ (j-2)!}
\bigg(\sum\limits_{p=1}^{n-i-j-1}\Big(\big(\sum\limits_{s=0}^{i-2}(-1)^{i+s}C^s_{i-2}\alpha_{p+s}\big)
\big(\sum\limits_{r=0}^{j-2}(-1)^{j+r}C^r_{j-2}\alpha_{p+i+r}\big)\\
- \big(\sum\limits_{r=0}^{j-2}(-1)^{j+r}C^r_{j-2}\alpha_{p+r}\big)\big(\sum\limits_{s=0}^{i-2}(-1)^{i+s}C^s_{i-2}\alpha_{p+j+s}\big)\Big)E_{p,p+i+j} \\
+ \Big(\sum\limits_{s=0}^{i-2}(-1)^{i+s}C^s_{i-2}\alpha_{n+s+1-i-j}-\sum\limits_{r=0}^{j-2}(-1)^{j+r}C^r_{j-2}\alpha_{n+r+1-i-j}\Big)E_{n+1-i-j,n}\bigg).
\end{multline*}

On the other hand
$$[\varphi(e_i),\varphi(e_j)]=(j-i)\varphi(e_{i+j})=\frac{(j-i)}{(i+j-2)!}\Big(\sum\limits_{p=1}^{n-i-j-1}\big(\sum\limits_{s=0}^{i+j-2}(-1)^{i+j+s}C^s_{i+j-2}\alpha_{p+s}\big)E_{p,p+i+j}+E_{n+1-i-j,n}\Big).$$

Next, we have the following system of equations
\begin{equation}\label{eq6}\begin{array}{l}
\big(\sum\limits_{s=0}^{i-2}(-1)^{i+s}C^s_{i-2}\alpha_{p+s}\big)
\big(\sum\limits_{r=0}^{j-2}(-1)^{j+r}C^r_{j-2}\alpha_{p+i+r}\big)-\big(\sum\limits_{r=0}^{j-2}(-1)^{j+r}C^r_{j-2}\alpha_{p+r}\big) \\[5mm]
\big(\sum\limits_{s=0}^{i-2}(-1)^{i+s}C^s_{i-2}\alpha_{p+j+s}\big)+\displaystyle\frac{(i-j)(i-2)! \ (j-2)!}{(i+j-2)!}\sum\limits_{s=0}^{i+j-2}(-1)^{i+j+s}C^s_{i+j-2}\alpha_{p+s}=0,
\end{array}\end{equation}
where $1\leq p\leq n-i-j-1, \quad i+j\leq n-2$.

\begin{equation}\label{eq7}\begin{array}{ll}\sum\limits_{r=0}^{j-2}(-1)^{j+r}C^r_{j-2}\alpha_{n+r+1-i-j}-\sum\limits_{s=0}^{i-2}(-1)^{i+s}C^s_{i-2}\alpha_{n+s+1-i-j}=\displaystyle\frac{(i-j)(i-2)! \ (j-2)!}{(i+j-2)!}, & i+j\leq n.
\end{array}\end{equation}

One of the solutions of the system of equations \eqref{eq6} and \eqref{eq7} is
$$\alpha_i= \frac{1}{n-i}, \quad 1\leq i\leq n-3.$$

Now we will check it. We using the next property of binomial coefficients
\begin{equation}\label{eq8}
\sum\limits_{k=0}^m \frac{(-1)^kC_m^k}{x+k}= \frac{m!}{x(x+1)\cdots(x+m)}, \qquad x \notin \{0,-1,\dots,-m\}.
\end{equation}

By putting all the values of $\alpha_i$ in the system \eqref{eq6}--\eqref{eq7}, and by using the property \eqref{eq8}, we get
\begin{multline*}
\frac{(i-2)!}{(n-i-p+2)(n-i-p+3)\cdots(n-p)} \cdot \frac{(j-2)!}{(n-i-j-p+2)(n-i-j-p+3)\cdots(n-i-p)}\\[4mm]
- \frac{(j-2)!}{(n-j-p+2)(n-j-p+3)\cdots(n-p)} \cdot \displaystyle\frac{(i-2)!}{(n-i-j-p+2)(n-i-j-p+3)\cdots(n-j-p)}\\[4mm]
+ \frac{(i-j)(i-2)! \ (j-2)!}{(n-i-j-p+2)(n-i-j-p+3)\cdots(n-p)}=0, \quad 1\leq p\leq n-i-j-1, \quad i+j\leq n-2,
\end{multline*}
and
\[
\frac{(j-2)!}{(i+1)(i+2)\cdots(i+j-1)}-\frac{(i-2)!}{(j+1)(j+2)\cdots(i+j-1)}=\frac{(i-j)(i-2)! \ (j-2)!}{(i+j-2)!}, \quad i+j\leq n.
\]
So, the values of $\alpha_i$ satisfy  the system of equations \eqref{eq6}--\eqref{eq7}.

From Proposition \ref{eq} we get that this representation is minimal.
\end{proof}

Now, we construct a module $V \times \mathcal{W}_{n} \rightarrow V,$ such that $$(x, e) =   x \varphi(e).$$ Then we obtain
\begin{equation*}\left\{\begin{array}{ll}
(x_{i}, e_1) = x_{i+1}, & 1 \leq i \leq n-2,\\[1mm]
(x_i, e_2) = \frac {1}{n-i}x_{i+2}, & 1 \leq i \leq n-3,\\[1mm]
(x_i, e_j)=\frac{1}{(j-2)!}\sum\limits_{s=0}^{j-2}\frac{(-1)^{j+s}C_{j-2}^s}{n-i-s}x_{i+j}, & 3 \leq j \leq n-2, \quad 1\leq i\leq n-j-1,\\[1mm]
(x_{n+1-j}, e_j) = \frac{1}{(j-2)!}x_n , & 2 \leq j \leq n,
\end{array}\right.\end{equation*}
and the remaining products in the action are zero.

\section{Leibniz algebras constructed by minimal faithful representations of Lie algebra}

Now we investigate Leibniz algebras $L$ such that  $L/ I \cong \mathcal{Q}_{2n}$ and $I=V$ as a $\mathcal{Q}_{2n}$-module.

Further we define the multiplications $[e_i, e_j]$ for $1
\leq i,j \leq 2n.$ We put
\begin{equation}\label{eq0}[e_i,e_j]=\left\{\begin{array}{ll}
e_{i+1} + \sum\limits_{k=1}^{2n} \alpha_{i,1}^k x_k, & i=1, \  2\leq j \leq 2n-2, \\[1mm]
-e_{j+1} + \sum\limits_{k=1}^{2n} \alpha_{1,j}^k x_k, & j=1, \ 2 \leq i \leq 2n-2, \\[1mm]
(-1)^ie_{2n}+\sum\limits_{k=1}^{2n}\alpha_{i,j}^kx_k, & i=2n-j+1, \ 2\leq j\leq n, \\[1mm]
(-1)^{i+1}e_{2n}+\sum\limits_{k=1}^{2n}\alpha_{i,j}^kx_k, & j=2n-i+1, \ 2\leq i\leq n, \\[1mm]
\sum\limits_{k=1}^{2n} \alpha_{i,j}^{k} x_k, & \text{otherwise}.
\end{array}\right.\end{equation}

In the multiplication \eqref{eq0}, by taking the basis transformation
$$\begin{array}{ll}
e_1' = e_1 - \sum\limits_{k=2}^{2n-2}\alpha_{1,1}^{k+1}x_{k} - (\alpha_{1,2}^{2n}+\alpha_{2,1}^{2n})x_{2n-1}, &
e_2' = e_2 -\sum\limits_{k=2}^{2n-2}(\alpha_{1,2}^{k+1}+\alpha_{2,1}^{k+1})x_{k},\\[1mm]
e_i^\prime=[e_1^\prime,e_{i-1}^\prime], \quad 3\leq i\leq 2n-1, & e_{2n}^\prime=[e_{2n-1}^\prime,e_2^\prime],
\end{array}$$
we obtain
\begin{equation}\label{eq2}\begin{array}{ll}
[e_1,e_1]= \alpha_{1,1}^{1} x_{1}+\alpha_{1,1}^{2} x_{2}+\alpha_{1,1}^{2n} x_{2n}, & [e_2,e_1]=-e_3+\alpha_{2,1}^{1} x_{1}+\alpha_{2,1}^{2} x_{2}, \\[2mm]
[e_1,e_i]=e_{i+1}, \ 2 \leq  i \leq 2n-2, & [e_{2n-1},e_2]=e_{2n}.
\end{array}\end{equation}

There are difficult to classify the general case, therefore we classify low-dimensional Leibniz algebras of such type. It is well known that $\mathcal{L}_4\cong \mathcal{Q}_4$, therefore we start classifying Leibniz algebras such that $L/ I \cong \mathcal{Q}_6$.

Using the multiplications \eqref{eq1}--\eqref{eq2}, and by checking Leibniz identity,  we get the following family of algebras denoted by
$\lambda(\alpha_1,\alpha_2,\alpha_3,\alpha_4,\alpha_5,\alpha_6,\alpha_7,\alpha_8,\alpha_9):$
\[\left\{\begin{array}{lll}
[e_1,e_1] = \alpha_1 x_6,  & [e_1,e_3]= e_4, & [x_1,e_6]= -2 x_6, \\[1mm]
[e_3, e_1] = -e_4, & [e_5, e_3] =\frac{1}{4}\alpha_3x_6, & [e_2,e_1]=-e_3+\alpha_2x_1+\alpha_3x_2, \\[1mm]
[e_4, e_1] = -e_5, & [x_1, e_3] = -x_3, & [e_2, e_2] = \alpha_5 x_3 + \alpha_7 x_4 + \alpha_8 x_ 5, \\[1mm]
[e_5, e_1] = -\alpha_4 x_6, & [x_4, e_3] =  x_6, & [e_3, e_2] = 4\alpha_2x_2-\alpha_6x_3 -2\alpha_7x_5-\alpha_9x_6, \\[1mm]
[e_6, e_1]=-\frac {1}{4}\alpha_3x_6, & [e_1,e_4] =e_5, & [e_4,e_2]=-2\alpha_2x_3+\frac{1}{2}\alpha_6x_4, \\[1mm]
[x_2,e_1]=x_3, & [e_3,e_4]=e_6, & [e_2,e_3]=-3\alpha_2 x_2 + \alpha_6x_3-\alpha_5x_4+\alpha_7 x_5 + \alpha_ 9 x_ 6, \\[1mm]
[x_3,e_1]=x_4, & [x_1,e_4]=x_4, & [e_3,e_3]=-2 \alpha_2x_3+\frac{1}{2}\alpha_6x_4, \\[1mm]
[x_4,e_1]= x_5, & [x_3, e_4] = x_6, & [e_4, e_3] = -e_6 + 2\alpha_2 x_4 - \frac{1}{2} \alpha_6 x_5, \\[1mm]
[e_1,e_2]=e_3, & [e_1,e_5] =\alpha_4 x_6, & [e_2, e_4]=4\alpha_2x_3-\frac{3}{2}\alpha_6x_4+\alpha_5x_5, \\[1mm]
[e_5, e_2]=e_6, & [x_1,e_5]=-x_5, & [e_4,e_4]=-2\alpha_2x_5-\frac{1}{2}\alpha_3x_6, \\[1mm]
[e_6,e_2]=-\alpha_6x_6, & [x_2,e_5] =x_6, & [e_2,e_5]=-e_6-3\alpha_2x_4+\frac{3}{2}\alpha_6x_5, \\[1mm]
[x_1, e_2]=x_2, & [e_2, e_6]=\frac{5}{2}\alpha_6x_6, & [e_3,e_5]=2\alpha_2x_5+\frac{3}{4}\alpha_3x_6, \\[1mm]
[x_5,e_2] =x_6, & [e_3,e_6]=-2\alpha_2x_6, & [e_1,e_6]=-2\alpha_2x_5-\frac{3}{4}\alpha_3x_6.
\end{array}\right.\]

\begin{teo}\label{t0}
Let $L$ be a $12$-dimensional Leibniz algebra such that $L/ I \cong \mathcal{Q}_6$ and
$I$ is a natural $L/ I$-module with a minimal faithful representation.
Then $L$ is isomorphic to the one of the pairwise non isomorphic algebras given in Appendix A.
\end{teo}

\begin{proof}
Let $L(\alpha) \coloneqq L$ be the 12-dimensional Leibniz algebra given by $\lambda(\alpha_1,\alpha_2,\alpha_3,\alpha_4,\alpha_5,\alpha_6,\alpha_7,\alpha_8,\alpha_9).$
Let $\varphi \colon L(\alpha)\rightarrow L(\alpha^\prime)$ be the isomorphism of Leibniz algebras:
$$\varphi(e_1)= \sum_{k=1}^{6} A_k e_k+\sum_{k=1}^{6} B_k x_k, \quad \varphi(e_2)= \sum_{k=1}^{6} P_k e_k+\sum_{k=1}^{6} Q_k x_k, \quad \varphi(x_1)= \sum_{k=1}^{6} M_k e_k+\sum_{k=1}^{6} R_k x_k,$$
and the other elements of the new basis are obtained as products of the above elements.

Then, we obtain the following restrictions:
$$A_1P_2R_4\neq0, \quad A_2=B_1=P_1=M_i=0, \quad 1\leq i\leq6,$$
$$A_6=\frac{-A_4^2+2A_3A_5}{2A_1}, \quad P_4=\frac{P_3^2}{2P_2}, \quad R_2=\frac{A_3R_1}{A_1}, \quad R_3=-\frac{A_4R_1}{A_1}, \quad R_4=\frac{A_5R_1}{A_1},$$

$$B_ 2 = \frac {2 A_ 3^2 \alpha_ 2} {A_ 1} ,\quad
B_ 3 = \frac {-4\alpha_ 2 A_ 3 A_ 4 - \alpha_ 6 A_ 3^2 } {2 A_ 1}, \quad
B_ 4 = \frac {2 \alpha_ 2 A_ 4^2 + \alpha_ 6 A_ 3 A_ 4 }{2 A_ 1},$$
$$B_5=\frac{4\alpha_2P_2^2(A_1^2R_6-A_4A_5R_1-A_1A_3R_5)-\alpha_6 A_4^2P_2^2R_1-4A_3^2\alpha_7P_2^2R_1}{4A_1P_2^2R_1}+$$
$$\frac{\alpha_3(4A_1A_5P_2P_3R_1-2A_1A_4P_3^2R_1+4A_1A_3P_2P_5R_1-4A_ 1^2P_2P_6R_1+4A_1^2P_2^2R_5)}{4A_1P_2^2R_1},$$
$$Q_ 1 = -\alpha_ 2 P_ 3, \quad Q_ 2 = \frac {3\alpha_2(A_3P_3-A_4P_2) - \alpha_ 3 A_ 1  P_ 3} {A_ 1},$$
$$Q_ 3 = \frac {2\alpha_ 2 (4 A_ 5 P_ 2^2 - A_ 4 P_ 2 P_ 3 - A_ 3 P_ 3^2) + \alpha_ 3A_ 1  P_ 3^2 + 2 \alpha_ 5A_ 3  P_ 2^2 +2 \alpha_ 6 (A_ 4 P_ 2^2 - 2 A_ 3 P_ 2 P_ 3)} {2 A_ 1 P_ 2}$$
$$Q_ 4 =\frac {\alpha_ 2 (2 A_ 4 P_ 3^2 R_ 1 +  A_ 1 P_ 2^2 R_ 5 - 3 A_ 5  P_ 2 P_ 3 R_ 1 + 2 A_ 1 P_ 2 P_ 6 R_ 1 -
2 A_ 3  P_ 2 P_ 5 R_ 1) } { A_ 1 P_ 2 R_ 1}+$$
$$\frac{- 4\alpha_ 3 A_ 1  P_ 2 P_ 5  - 4 \alpha_ 5 A_ 4  P_ 2^2  +\alpha_ 6 (A_ 3  P_ 3^2  - 6 A_ 5 P_ 2^2  +  2 A_ 4 P_ 2 P_ 3) +   4 \alpha_ 7 A_ 3 P_ 2^2 } {4 A_ 1 P_ 2 },$$
$$Q_ 5 = \frac { \alpha_ 2 (A_ 5 P_ 2^2 P_ 3^2 R_ 1 -
       A_ 4  P_ 2 P_ 3^3 R_ 1 + 2 A_ 3  P_ 2^2 P_ 3 P_ 5 R_ 1 -
       2 A_ 1 P_ 2^2 P_ 3 P_ 6 R_ 1 -
       A_ 1 P_ 2^3 P_ 3 R_ 5)} {A_ 1 P_ 2^3 R_ 1} +$$
$$\frac {8 \alpha_ 5 A_ 5  P_ 2^4 R_ 1 + \alpha_ 6 (12 A_ 3 P_ 2^3 P_ 5 R_ 1 \
- 8 A_ 4  P_ 2^2 P_ 3^2 R_ 1 - 12 A_ 1 P_ 2^3 P_ 6 R_ 1 +
      12 A_ 5 P_ 2^3 P_ 3 R_ 1)} {8 A_ 1 P_ 2^3 R_ 1}+$$
$$\frac{\alpha_ 3 A_ 1  ( P_ 2^2 P_ 3 P_ 5 - P_ 3^4 ) +\alpha_ 7 (
     A_ 4 P_ 2^4  - 2 A_ 3  P_ 2^3 P_ 3 ) +
   \alpha_ 8 A_ 3  P_ 2^4 }{ A_ 1 P_ 2^3 },$$

And
$$\alpha_1^\prime=\displaystyle\frac{\alpha_1}{A_1P_2^2 R_1}, \quad \alpha_2^\prime=\displaystyle\frac{\alpha_2A_1P_2}{R_1}, \quad \alpha_3^\prime=\displaystyle\frac{\alpha_3A_1}{R_1}, \quad \alpha_4^\prime=\displaystyle\frac{\alpha_4 A_1}{P_2R_1}, \quad
\alpha_5^\prime=\displaystyle\frac{\alpha_5P_2}{A_1R_1}, \quad \alpha_6^\prime=\displaystyle\frac{\alpha_6 P_2}{R_1},$$
$$\alpha_7^\prime=\displaystyle\frac{2\alpha_7P_2^3+\alpha_2(P_3^3-6P_2^2P_5)}{2A_1^2P_2^2R_1}, \quad
\alpha_8^\prime=\displaystyle\frac{4\alpha_8P_2^3-\alpha_6(P_3^3-6P_2^2P_5)}{4A_1^3P_2^2R_1},$$
$$\alpha_9^\prime=\frac{\alpha_9}{A_1^2R_1}+\frac{2\alpha_2(2A_5P_2^2P_3R_1-A_4P_2P_3^2R_1-2A_1P_2^2P_6R_1
+2A_1P_2^3R_5+2A_3P_2^2P_5R_1)}{A_1^3P_2^3R_1^2}+\frac{\alpha_3(P_3^3-6P_2^2P_5)}{8A_1^2P_2^3R_1}$$
$$$$
Considering all the possible cases, we obtain the families of algebras listed in the theorem.
\end{proof}

Now we give the classification of Leibniz algebras $L$ such that $L/ I \cong \mathcal{W}_{5}$ and $L/ I \cong \mathcal{R}_{7}$.
We denote the next families of algebras by $\mu(\gamma_1,\gamma_2,\gamma_3,\gamma_4,\gamma_5,\gamma_6,\gamma_7)$ and $\eta(\beta_1,\beta_2,\beta_3,\beta_4)$:

\[\left\{ \begin{array}{lll}
[e_ 1, e_ 1]=\gamma_1x_5, & [e_2,e_1]=-e_3, & [e_3,e_1]=-2e_4, \\[1 mm]
[e_4,e_1]=-3e_5, & [e_5,e_1]=-\gamma_2x_5, & [e_2,e_4]=\frac{1}{2}(\gamma_3x_4+\gamma_2x_5),\\[1 mm]
[x_2,e_1]=x_3, & [x_3,e_1]=x_4, & [e_1,e_2]=e_3, \\[ 1 mm]
[x_2,e_2]=\frac{1}{3}x_4, &  [x_3,e_3]=x_5, & [e_4,e_2]=-\frac{1}{2}\gamma_2x_5,\\[1 mm]
[e_5,e_2]=-\gamma_5x_5, & [x_1,e_2]=\frac{1}{4}x_3, &  [e_2,e_2]=\gamma_3x_2+\gamma_4x_3+\gamma_6x_4, \\[1 mm]
[x_4,e_2]=x_5, & [e_1,e_3]=2e_4, & [e_2,e_3]=e_5-\gamma_3x_3+\gamma_4x_4+\gamma_7x_5, \\[1 mm]
[e_4,e_3]=3\gamma_5x_5, & [x_1,e_3] = \frac{1}{12}x_4, & [e_3,e_2]=-e_5-2\gamma_4x_4-\gamma_7x_5, \\[1 mm]
[e_1,e_4]=3e_5, & [x_1,e_1]=x_2,  & [e_3,e_4]=-3\gamma_5x_5, \\[1 mm]
[x_2,e_4]=\frac{1}{2}x_5, & [e_1,e_5]=\gamma_2x_5, &  [e_2,e_5]=\gamma_5x_5, \\[1 mm]
[x_1,e_5]=\frac{1}{6}x_5.
\end{array}\right.\]
and
\[\left\{ \begin{array}{llll}
[e_ 1, e_ 1] = \beta_ 1 x_ 7, & [e_ 1, e_ 2] = e_ 3, & [e_ 1, e_ 3] = e_ 4, & [e_ 1, e_ 4] = e_ 5,\\[1mm]
[e_ 1, e_ 5] = e_ 6, & [e_ 1, e_ 6] = e_ 7, & [e_ 1, e_ 7] = \beta_ 2 x_ 7, & [e_ 2, e_ 1] = -e_ 3, \\[1mm]
[e_ 2, e_ 2] = \beta_ 3 x_ 4 + \beta_ 4 x_ 6, & [e_ 2, e_ 3] = e_ 5 - \beta_ 3 x_ 5, & [e_ 2, e_ 4] = e_ 6 + \beta_ 3 x_ 6, &
[e_ 2, e_ 5] = e_ 7, \\[1mm]
[e_ 2, e_ 6] = \beta_ 2 x_ 7, & [e_ 3, e_ 1] = -e_ 4, & [e_ 3, e_ 2] = -e_ 5, & [e_ 4, e_ 1] = -e_ 5, \\[1mm]
[e_ 4, e_ 2] = -e_ 6, & [e_ 5, e_ 1] = -e_ 6, & [e_ 5, e_ 2] = -e_ 7, & [e_ 6, e_ 1] = -e_ 7, \\[1mm]
[e_ 6, e_ 2] = -\beta_ 2 x_ 7, & [e_ 7, e_ 1] = -\beta_ 2 x_ 7, & [x_ 1, e_ 1] = x_ 2, & [x_ 1, e_ 2] = x_ 3, \\[1mm]
[x_ 1, e_ 7] = x_ 7, & [x_ 2, e_ 1] = x_ 3, & [x_ 2, e_ 2] = x_ 4, & [x_ 2, e_ 6] = x_ 7, \\[1mm]
[x_ 3, e_ 1] = x_ 4, & [x_ 3, e_ 2] = x_ 5, & [x_ 3, e_ 5] = x_ 7, & [x_ 4, e_ 1] = x_ 5, \\[1mm]
[x_ 4, e_ 2] = x_ 6, & [x_ 4, e_ 4] = x_ 7, & [x_ 5, e_ 1] = x_ 6 , & [x_ 5, e_ 3] = x_ 7, \\[1mm]
[x_ 6, e_ 2] = x_ 7.
\end{array}\right.\]

\begin{teo}\label{t1}
Let $L$ be a $10$-dimensional Leibniz algebra such that $L/ I \cong \mathcal{W}_5$ and
$I$ is a natural $L/ I$-module with a minimal faithful representation.
Then $L$ is isomorphic to the one of the pairwise non isomorphic algebras given in Appendix B.
\end{teo}

\begin{teo}\label{t2}
Let $L$ be a $14$-dimensional Leibniz algebra such that $L/ I \cong \mathcal{R}_7$ and
$I$ is a natural $L/ I$-module with a minimal faithful representation.
Then $L$ is isomorphic to the one of the following pairwise non isomorphic algebras:
$$\begin{array}{lllll}
\eta(0,0,0,0,0),&\eta(0,0,0,1),&\eta(0,0,1,0),&\eta(0,1,0,1),&\eta(0,1,\beta_3,0)_{\beta_3\neq0},\\[1mm]
\eta(1,0,0,0,0),&\eta(1,0,0,1),&\eta(1,0,1,0),&\eta(1,1,0,\beta_4),&\eta(1,1,\beta_3,0)_{\beta_3\neq0},\\[1mm]
\end{array}$$
with $\beta_3,\beta_4\in \mathbb{C}.$
\end{teo}

The proofs of Theorem \ref{t1} and Theorem \ref{t2} are carried out by applying arguments used in Theorem \ref{t0}.

\newpage
\appendix
\section{First Appendix}

\[\begin{array}{|c|c|c|c|}
\hline
\lambda(0,0,0,0,0,0,0,0,0)&\lambda(0,0,0,0,0,0,0,0,1)&\lambda(0,0,0,0,0,0,0,1,0)&\lambda(0,0,0,1,1,1,\alpha_7,0,\alpha_9)\\
\hline
\lambda(0,0,0,0,0,0,1,0,0)&\lambda(0,0,0,0,0,0,1,0,1)&\lambda(0,0,0,0,0,0,1,1,0)&\lambda(1,0,0,0,1,0,1,\alpha_8,\alpha_9)\\
\hline
\lambda(0,0,0,0,0,1,0,0,0)&\lambda(0,0,0,0,0,1,0,0,1)&\lambda(0,0,0,0,0,1,1,0,0)&\lambda(1,0,1,0,1,0,\alpha_7,\alpha_8,0)\\
\hline
\lambda(0,0,0,0,1,0,0,0,0)&\lambda(0,0,0,0,1,0,0,0,1)&\lambda(0,0,0,0,1,0,0,1,0)&\lambda(0,0,1,1,0,0,1,\alpha_8,0)\\
\hline
\lambda(0,0,0,0,0,0,0,1,1)&\lambda(0,0,0,1,0,0,0,1,1)&\lambda(0,0,0,0,1,1,\alpha_7,0,0)&\lambda(0,0,0,0,1,1,\alpha_7,0,1)\\
\hline
\lambda(0,0,0,1,0,0,0,0,0)&\lambda(0,0,0,1,0,0,0,0,1)&\lambda(0,0,0,1,0,0,0,1,0)&\lambda(0,0,0,0,1,0,1,\alpha_8,1)\\
\hline
\lambda(0,0,0,1,0,0,1,0,0)&\lambda(0,0,0,1,0,0,1,0,1)&\lambda(0,0,0,1,0,0,1,1,\alpha_9)&\lambda(1,0,0,1,1,0,\alpha_7,\alpha_8,\alpha_9)\\
\hline
\lambda(0,0,0,1,0,1,0,0,1)&\lambda(0,0,0,1,0,1,0,0,0)&\lambda(0,0,0,1,1,0,0,0,0)&\lambda(0,0,0,1,1,0,1,\alpha_8,\alpha_9)\\
\hline
\lambda(0,0,0,0,0,0,1,1,1)&\lambda(0,0,0,1,1,0,0,0,1)&\lambda(0,0,0,0,1,0,1,\alpha_8,0)&\lambda(0,0,1,0,0,1,1,\alpha_8,0)\\
\hline
\lambda(0,0,1,0,0,0,0,1,0)&\lambda(0,0,1,0,0,0,1,0,0)&\lambda(0,0,1,0,0,0,1,1,0)&\lambda(0,0,1,0,1,1,\alpha_7,\alpha_8,0)\\
\hline
\lambda(0,0,1,0,0,1,0,1,0)&\lambda(0,0,1,0,0,0,0,0,0)&\lambda(0,0,1,0,1,0,0,0,0)&\lambda(0,0,1,1,0,1,\alpha_7,0,0)\\
\hline
\lambda(0,0,0,0,0,1,1,0,1)&\lambda(0,0,1,0,0,1,0,0,0)&\lambda(0,0,1,1,0,0,0,0,0)&\lambda(1,0,1,1,\alpha_5,0,\alpha_7,\alpha_8,0)\\
\hline
\lambda(0,0,0,0,1,0,0,1,1)&\lambda(0,0,1,0,1,0,0,1,0)&\lambda(1,0,0,0,0,0,1,1,\alpha_9)&\lambda(0,0,1,1,1,0,\alpha_7,\alpha_8,0)\\
\hline
\lambda(0,1,0,0,0,0,0,0,0)&\lambda(0,1,0,0,0,0,0,1,0)&\lambda(0,1,0,0,0,1,0,\alpha_8,0)&\lambda(0,1,0,0,1,\alpha_6,0,\alpha_8,0)\\
\hline
\lambda(0,1,0,1,0,0,0,0,0)&\lambda(0,1,0,1,0,0,0,1,0)&\lambda(0,1,0,1,0,1,0,\alpha_8,0)&\lambda(0,1,0,1,1,\alpha_6,0,\alpha_8,0)\\
\hline
\lambda(0,1,1,0,0,0,0,0,0)&\lambda(0,1,1,0,0,0,0,1,0)&\lambda(0,1,1,0,0,1,0,\alpha_8,0)&\lambda(0,1,1,0,1,\alpha_6,0,\alpha_8,0)\\
\hline
\lambda(0,0,0,1,0,0,0,1,1)&\lambda(1,0,0,0,0,0,0,0,0)&\lambda(1,0,0,0,1,0,0,1,\alpha_9)&\lambda(1,0,1,1,\alpha_5,\alpha_6,\alpha_7,0,0)\\
\hline
\lambda(1,0,0,0,0,0,0,1,1)&\lambda(1,0,0,0,0,0,1,0,0)&\lambda(1,0,1,0,0,1,\alpha_7,0,0)&\lambda(0,0,1,1,1,\alpha_6,\alpha_7,0,0)_{\alpha_6\neq0}\\
\hline
\lambda(1,0,0,0,0,1,0,0,0)&\lambda(1,0,0,0,0,1,0,0,1)&\lambda(1,0,0,0,0,1,1,0,\alpha_9)&\lambda(1,1,0,1,\alpha_5,\alpha_6,0,\alpha_8,0)\\
\hline
\lambda(1,0,0,0,1,0,0,0,1)&\lambda(0,0,1,1,0,0,0,1,0)&\lambda(0,0,0,1,1,0,0,1,\alpha_9)&\lambda(1,0,0,0,1,1,\alpha_7,0,\alpha_9)\\
\hline
\lambda(1,0,0,1,0,0,0,0,0)&\lambda(1,0,0,1,0,0,0,0,1)&\lambda(1,0,0,1,0,0,0,1,\alpha_9)&\lambda(1,0,0,1,0,0,1,\alpha_8,\alpha_9)\\
\hline
\lambda(1,0,1,0,0,0,0,0,0)&\lambda(1,0,1,0,0,0,0,1,0)&\lambda(0,0,0,1,0,1,1,0,\alpha_9)&\lambda(1,0,0,1,1,\alpha_6,\alpha_7,0,\alpha_9)_{\alpha_6\neq0}\\
\hline
\lambda(1,0,0,0,1,0,0,0,0)&\lambda(1,0,0,0,0,0,1,0,1)&\lambda(0,0,1,0,1,0,1,\alpha_8,0)&\lambda(1,0,1,0,1,\alpha_6,\alpha_7,0,0)_{\alpha_6\neq0}\\
\hline
\lambda(1,0,0,0,0,0,0,1,0)&\lambda(1,0,0,0,0,0,0,0,1)&\lambda(1,1,0,0,0,1,0,\alpha_8,0)&\lambda(1,1,0,0,1,\alpha_6,,\alpha_8,0)\\
\hline
\lambda(1,1,0,0,0,0,0,0,0)&\lambda(1,1,0,0,0,0,0,1,0)&\lambda(1,0,1,0,0,0,1,\alpha_8,0)&\lambda(1,1,1,\alpha_4,\alpha_5,\alpha_6,0,\alpha_8,0)\\
\hline
\end{array}\]
with $\alpha_4,\alpha_5,\alpha_6,\alpha_7,\alpha_8,\alpha_9\in \mathbb{C}.$

\section{Second Appendix}

\[\begin{array}{|c|c|c|c|}
\hline
\mu(0,0,0,0,0,0,0)&\mu(0,0,1,0,1,0,0)&\mu(0,1,0,0,0,\gamma_6,0)&\mu(0,0,1,1,\gamma_5,0,0)_{\gamma_5\neq0}\\
\hline
\mu(0,0,0,0,0,0,1)&\mu(1,0,0,0,0,0,0)&\mu(1,0,0,0,1,\gamma_6,0)&\mu(0,1,0,1,\gamma_5,\gamma_6,0)_{\gamma_5\neq0}\\
\hline
\mu(0,0,0,0,0,1,0)&\mu(1,0,0,0,0,0,1)&\mu(0,1,0,1,0,\gamma_6,\gamma_7)&\mu(0,1,1,\gamma_4,\gamma_5,0,0)_{\gamma_5\neq0}\\
\hline
\mu(0,0,0,0,0,1,1)&\mu(0,0,0,1,0,1,\gamma_7)&\mu(1,0,0,1,0,\gamma_6,\gamma_7)&\mu(1,0,0,1,\gamma_5,\gamma_6,0)_{\gamma_5\neq0}\\
\hline
\mu(0,0,0,0,1,0,0)&\mu(0,0,0,1,\gamma_5,1,0)&\mu(1,0,1,\gamma_4,0,0,\gamma_7)&\mu(1,0,1,\gamma_4,\gamma_5,0,0)_{\gamma_5\neq0}\\
\hline
\mu(0,0,0,0,1,1,0)&\mu(0,0,1,1,0,0,\gamma_7)&\mu(0,1,1,\gamma_4,0,0,\gamma_7)&\mu(1,1,0,\gamma_4,\gamma_5,\gamma_6,0)_{\gamma_5\neq0}\\
\hline
\mu(0,0,0,1,0,0,1)&\mu(0,1,0,0,0,\gamma_6,1)&\mu(1,1,0,\gamma_4,0,\gamma_6,\gamma_7)&\mu(1,1,\gamma_3,\gamma_4,0,0,\gamma_7)_{\gamma_3\neq0}\\
\hline
\mu(0,0,1,0,0,0,0)&\mu(1,0,0,0,0,1,\gamma_7)&\mu(0,0,0,1,\gamma_5,0,0)_{\gamma_5\neq0}&\mu(1,1,\gamma_3,\gamma_4,\gamma_5,0,0)_{\gamma_3\neq0,\gamma_5\neq0}\\
\hline
\mu(0,0,1,0,0,0,1)&\mu(0,1,0,0,1,\gamma_6,0)&&\\
\hline
\end{array}\]
with $\gamma_3,\gamma_4,\gamma_5,\gamma_6,\gamma_7\in \mathbb{C}.$

\newpage


\end{document}